\documentclass{article}
\usepackage{amssymb,amsmath,latexsym, amsthm,enumerate,  amsfonts}
\usepackage{wrapfig, tikz}
\usetikzlibrary{matrix, calc}
\usepackage{blindtext}
\title{} \author{} \date{}
\usepackage[pdftex, pdfstartview=FitH]{hyperref} % Use LaTeX and then DVI2PDF,
\usepackage[paperwidth=210mm, paperheight=297mm, twoside, hmargin={25mm,20mm}, vmargin={20mm,20mm} ]{geometry}

\numberwithin{equation}{section} %You may omit this line if you want numbering as (1)

\newcommand{\Cl}{\mathop{\mathrm {Cl}}\nolimits}
\newcommand{\Int}{\mathop{\mathrm {Int}}\nolimits}
\newcommand{\id}{\mathop{\mathrm {id}}\nolimits}

 \newtheorem{thm}{Theorem}
 \newtheorem{cor}{Corollary}
 
 \newtheorem{prop}{Proposition}
 \theoremstyle{definition}
 
  \newtheorem{ex}{Example}
 \theoremstyle{remark}
 \newtheorem{rem}{Remark}

\begin{document}
\thispagestyle{empty}
% PRAVLJENJE NASLOVA
\begin{center}
{\large \bf  On preserving continuity in ideal topological spaces\footnote{The authors acknowledge financial support of the Ministry of Education, Science and Technological Development of the Republic of Serbia (Grant No. 451-03-9/2021-14/ 200125)
}

} \vspace*{3mm}

{\bf Anika Njamcul\footnote{Department of Mathematics and Informatics, Faculty of Sciences, University of Novi Sad, Serbia, e-mail: \href{mailto:anika.njamcul@dmi.uns.ac.rs}{anika.njamcul@dmi.uns.ac.rs}}}
and
{\bf Aleksandar Pavlovi\'c\footnote{Department of Mathematics and Informatics, Faculty of Sciences, University of @Novi Sad, Serbia, e-mail: \href{mailto:apavlovic@dmi.uns.ac.rs}{apavlovic@dmi.uns.ac.rs}}}
\end{center}

\begin{abstract}

We present some sufficient conditions for continuity of the mapping $f:\langle X,\tau_X^*\rangle \to \langle Y,\tau_Y^*\rangle$, where $\tau_X^*$ and $\tau_Y^*$ are topologies induced by the local function on $X$ and $Y$, resp. under the assumption that the mapping from  $\langle X, \tau_X \rangle$ to $\langle Y, \tau_Y \rangle$ is continuous. Further, we consider open and closed functions in this matter, as we state the cases in which the open (or closed) mapping is being preserved through the "idealisation" of both domain and codomain.
Through several examples we illustrate that the  conditions we considered  can not be weakened.
\\[2mm] {\it AMS Mathematics  Subject Classification $(2010)$}:
54A10, %Several topologies on one set (change of topology, comparison of topologies, lattices of topologies)
54A05,  %Topological spaces and generalizations (closure spaces, etc.)
54B99, %None of the above, but in this section
54E99 %None of the above, but in this section
\\[1mm] {\it Key words and phrases:} ideal topological space, local function, continuous function, open mapping, closed mapping

\end{abstract}

\section{Introduction}

We can say that ideals and continuity belong to the  folklore in general topology, and, therefore,  in mathematics. First steps in introducing topological spaces enhanced by an ideal is due to Kuratowski \cite{KURold, KUR}  in 1933, who introduced local function as a generalization to closure. A little bit later ideals in topological spaces were studied by Vaidyanathaswamy \cite{Vaid}.  Freud \cite{Freud} generalized the Cantor-Bendixson theorem using ideal topological space.   Scheinberg \cite{Sein} applied ideals in the measure theory. Finally,  in 1990  Jankovi\'c and Hamlett \cite{JH} wrote a survey paper  on the topic of ideal topological spaces in which they gathered their recent results together with previous results in the area. Today this paper is a starting point, and  a pattern for introducing many variations and generalizations of open sets.
%Some of them can be found in   papers of Jafari  and  Rajesh \cite{JR}, and Manoharan and   Thangavelu \cite{MT}.

Continuity belongs to the origins of the general topology. Similar thing is with open and closed mappings. It is hard to imagine any topological course without the notion of continuous, open and closed function. Therefore, telling the story on history of such mappings we will leave to Engelking \cite{eng}.

The aim of this paper is to investigate preserving continuity, open, and closed mappings after changing original topology with the ideal applying the local function. Since such topology is always finer than the original, changing applying ideal only on domain topology will preserve continuity, and if we apply ideal on codomain topology, open (closed) mappings will remain open (closed). So, we investigate situations when both, domain and codomain, topologies are changed by the ideal and its local function.

\section{Definitions and notations}

We will use the following notation. If  $\langle X, \tau\rangle$ is  a topological space,  $\tau(x)$ will be the family of open neighbourhoods at the point $x$, $\Cl_{\tau}(A)$ or $\overline{A}^{\tau}$  the closure of the set $A$, and $\Int_{\tau}(A)$ its interior. If there is no chance of misunderstanding about the topology dealt with, we will use only $\Cl(A)$, $\overline{A}$ and $\Int(A)$.  If not stated otherwise, no separation axioms will be presumed, but all examples will be  in spaces which are at least $T_0$.

If $f:X \to Y$, for $A \subseteq X$ and $B \subseteq Y$, direct image of the set $A$ is defined by $f[A]=\{f(x):x \in A\}$ and preimage of $B$ is defined  by $f^{-1}[B]=\{x \in X :f(x) \in B \}$.\\

Just for the completeness of this paper we state several equivalent conditions for continuity of a function $f:X \to Y$ which will be used in the paper.

\begin{prop}\label{continuityeq}
\cite[Proposition 1.4.1]{eng} For  $f:\langle X, \tau_X\rangle \to \langle Y, \tau_Y\rangle$ the following conditions are equivalent

a) $f$ is continuous.

b) For each $O\in  \tau_Y$ we have $f^{-1}[O]\in \tau_X$.

c) For each $A\subseteq X$ we have $f[\overline{A}]\subseteq \overline{f[A]}$.

d) For each $B\subseteq Y$ we have $\overline{f^{-1}[B]}\subseteq {f^{-1}[\overline{B}]}$.

e) For each $B\subseteq Y$ we have ${f^{-1}[\Int(B)]}\subseteq \Int({f^{-1}[{B}]})$.

\end{prop}

We recall that $f:\langle X, \tau_X\rangle \to \langle Y, \tau_Y\rangle$ is an open (closed) mapping iff for each open (closed) $A  \subset X$ the set $f[A]$ is open (closed). Also, $f$ is a homeomorphism iff it is an open (closed) continuous bijection.

%For a set $A$ and a cardinal $\kappa$,   $[A]^{\kappa}$, $[A]^{<\kappa}$ and $[A]^{\leq \kappa}$  are  the families of all subsets of $A$ of cardinality $\kappa$, less than $\kappa$ and less than or equal to $\kappa$, respectively.  $\omega$   will denote the set of natural numbers including $0$.

If $X$ is a nonempty set, a family $\mathcal{I}\subset P(X)$ satisfying

{\leftskip 5mm (I0) $\emptyset \in \mathcal{I}$,

(I1) If $A \in  \mathcal{I}$  and $B \subseteq A$, then $B \in \mathcal{I}$,

(I2) If $A, B \in \mathcal{I}$, then $A \cup B \in \mathcal{I}$,

}

\noindent is called an \textbf{ideal} on $X$. If $X \not \in \mathcal{I}$ (i.e.,\ $P(X) \neq \mathcal{I}$), then $\mathcal{I}$ is \textbf{proper}. If there exists $A \subseteq X$ such that $B \in \mathcal{I}$ if $B \subseteq A$, then $\mathcal{I}$ is a \textbf{principal ideal}. The ideal of all finite subsets of a set $X$ will be denoted by $Fin(X)$.\\

 %Let $\langle X, \tau\rangle$ be a topological space.   $Fin$ is   the \textbf{ideal of finite sets}. $\mathcal{I}_{ctble}$ is the \textbf{ideal of countable sets}.   $\mathcal{I}_{cd}$ denotes the \textbf{ideal of closed discrete sets}. The family of discrete sets does not form  an ideal (the union of two discrete sets does not have to be discrete). A set $S$ is \textbf{scattered}  if each nonempty subset of $S$ contains an isolated point. If $X$ is a $T_1$ space, then the family of scattered sets is   the \textbf{ideal of scattered sets}, and it  is denoted by $\mathcal{I}_{sc}$. Every discrete set is scattered. A set $A$ is \textbf{relatively compact} if $\Cl(A)$ is compact. The family of all relatively compact sets forms the ideal $\mathcal{I}_{K}$. A set $A$ is \textbf{nowhere dense} if $\Int(\Cl(A))=\emptyset$. A countable union of nowhere dense sets is called a \textbf{meager set}. The family of nowhere dense sets forms the ideal $\mathcal{I}_{nwd}$, and the family of meager sets forms the ideal $\mathcal{I}_{mg}$.

If $\langle X, \tau\rangle$ is a topological space and $\mathcal{I}$ an ideal on $X$, then a triple $\langle X, \tau, \mathcal{I}\rangle$ is called an \textbf{ideal topological space}.

If $\langle X, \tau, \mathcal{I}\rangle$ is an ideal topological space, then the mapping $A \mapsto A^*_{(\tau, \mathcal{I})}$ (briefly $A^*$) defined by
$$A^*_{(\tau, \mathcal{I})}=\{x \in X: A \cap U \not \in \mathcal{I} \mbox { for each }  U \in \tau(x)\}$$
is called the \textbf{local function} (see \cite{KUR}).

The local function has the following properties (see \cite{JH}):

{\leftskip 5mm (1) $A \subseteq B \Rightarrow A^* \subseteq B^*$;

(2) $A^*=\Cl(A^*)\subseteq \Cl(A)$;

(3) $(A^*)^*\subseteq A^*$;

(4) $(A\cup B)^*=A^* \cup B^*$

(5) If $I \in \mathcal{I}$, then $(A\cup I)^*=A^*=(A\setminus I)^*$.

}

Also by $\Cl^*(A)=A \cup A^*$   a closure operator on $P(X)$ is defined and it generates a topology $\tau^*(\mathcal{I})$  (briefly $\tau^*$) on $X$ where $$\tau^*(\mathcal{I})=\{U \subseteq X: \Cl^*(X \setminus U)=X \setminus U\}.$$
So, set $A$  is closed in $\tau^*$ iff $A^* \subseteq A$. Using the local function, another operator $\Psi$ is defined by
$$\Psi(A)=X \setminus (X \setminus A)^*.$$
A set $A$ is in $\tau^*$ iff $A\subseteq \Psi(A) $. Also  $\Psi(A)$ is always open in $\tau$.

It is easy to see that $\tau \subseteq \tau^*\subseteq P(X)$.

For more details on the local function, $\Psi$ operator and topology $\tau^*$ we refer the reader to the papers of Jankovi\'c and Hamlett \cite{JH} and Hamlett and Rose \cite{HRStar}.

\section{Previous results}

One of the first results in preserving continuity is due to Samuels.

\begin{thm}[\cite{Sam}]
If $X=X^*$ and $Y$ is regular then $f:\langle X, \tau\rangle \to Y$ is continuous iff $f:\langle X, \tau^*\rangle \to Y$ is continuous.
\end{thm}

The following result by Natkaniec is a direct consequence of the Samuels's one, but, due to the  completely different technique of the proof and its interpretation in the universe of Polish spaces, it is worth mentioning.

\begin{thm}[\cite{Natkaniec1986}]  Let $f:  X  \to \mathbb{R} $, where $X$ is a Polish space with topology $\tau$, and $\mathcal{I}$  an $\sigma$-complete ideal on $X$ such that $Fin \subset \mathcal{I}$ and $\mathcal{I} \cap \tau=\{\emptyset\}$. If $f:\langle X, \tau^*\rangle \to \langle R, \mathcal{O}_{nat}\rangle$ is a continuous function, then $f:\langle X, \tau\rangle \to \langle R, \mathcal{O}_{nat}\rangle$ is also continuous.
\end{thm}

Finally, Hamlett and Jankovi\'c in 1990. brought several results on consequences of continuity in  ideal topological spaces given in terms of the $\Psi$ operator and  the local function. Let us mention that $\langle X,\tau, \mathcal{I}\rangle$ is $\mathcal{I}$-compact \cite{newcomb, rancin} iff for each open cover $\{U_\lambda : \lambda \in \Lambda\}$ exists a finite subcollection $\{U_{\lambda_k}: k \leq n\}$ such that $X \setminus \bigcup \{U_{\lambda_k}: k \leq n\} \in \mathcal{I}$.

\begin{thm}[\cite{JHComp}] Let $f:\langle X, \tau, \mathcal{I}\rangle \to \langle Y, \sigma, f[\mathcal{I}]\rangle$ be a bijection such that $\langle X, \tau\rangle$ is $\mathcal{I}$-compact and $\langle Y, \sigma\rangle$ is Hausdorff. If $f:\langle X, \tau^*\rangle \to \langle Y, \sigma\rangle$ is continuous, then $f:\langle X, \tau^*\rangle \to \langle Y, \sigma^*\rangle$ is a homeomorphism.
\end{thm}

We will say that topology $\tau$ and ideal $\mathcal{I}$ are compatible, denoted by $\mathcal{I} \sim \tau$, (see \cite{NJAST}) iff the following statement holds: If for each $A \subset X$ such that for each $x \in A$ there exists $U \in \tau(x)$  such that $U \cap A \in \mathcal{I} $ then $A \in \mathcal{I}$. On the other words, if $A$ is small locally, than it is small globally.

 Also, we will denote  $\Psi(\tau)=\{\Psi(U): U \in \tau\}$ and $\langle \Psi(\tau) \rangle $ will be  the smallest topology containing $\Psi(\tau)$.\\

Some relevant previous results concerning continuous functions are given here.

\begin{thm} [\cite{HRStar}] Let $\langle X, \tau, \mathcal{I}\rangle$ and $ \langle Y, \sigma, \mathcal{J} \rangle$  be ideal topological spaces. Let $f:\langle X, \tau\rangle \to \langle Y, \langle \Psi(\sigma)\rangle\rangle$ be a continuous injection, $\mathcal{J} \sim \sigma$ and $f^{-1}[\mathcal{J}]\subset \mathcal{I}$. Then $\Psi(f[A]) \subseteq f[\Psi(A)]$, for each $A \subseteq X$.
\end{thm}

\begin{thm} [\cite{HRStar}] Let $\langle X, \tau, \mathcal{I}\rangle$ and  $\langle Y, \sigma, \mathcal{J} \rangle$  be ideal topological spaces. Let $f:\langle X, \langle \Psi(\tau)\rangle \rangle \to \langle Y, \sigma \rangle$ be an open bijection, $\mathcal{I} \sim \tau$ and $f[\mathcal{I}]\subset \mathcal{J}$. Then $f[\Psi(A)] \subseteq \Psi(f[A])$, for each $A \subseteq X$.
\end{thm}

\begin{thm} [\cite{HRStar}] \label{homeo} Let $\langle X, \tau, \mathcal{I}\rangle \to \langle Y, \sigma, \mathcal{J} \rangle$ be ideal topological spaces. Let $f:  X \to \  Y $ be a  bijection  and $f[\mathcal{I}]= \mathcal{J}$. Then the following conditions are equivalent

a) $f:\langle X, \tau^*\rangle \to \langle Y, \sigma^*\rangle$ is a homeomorphism;

b) $f[A^*]=(f[A])^*$, for each $A \subseteq X$;

 c) $ f[\Psi(A)] = \Psi(f[A])$, for each $A \subseteq X$.
\end{thm}

Recently, at this topic also worked: Kuyucu,  Noiri and Özkurt \cite{KNO}; Özkurt \cite{MR2176209}; Ekici \cite{MR3098388}; 	 Al-Omari and Noiri \cite{MR3252455};  	Vadivel  and Navuluri \cite{MR3595518};  Goyal and Noorie \cite{GoyalNoorie}.

\section{Local function and continuity}

In this section we will start with continuous mappings with original topologies and we will give some sufficient condition when  continuity is preserved in spaces with idealized topologies.

\begin{thm} \label{Tc1} Let $\langle X, \tau_X, \mathcal{I}_X\rangle$ and $\langle Y, \tau_Y, \mathcal{I}_Y\rangle$ be ideal topological spaces. If $f:\langle X, \tau_X\rangle \to \langle Y, \tau_Y\rangle$ is a continuous function and for all $I\in \mathcal{I}_Y$ we have $f^{-1}[I]\in \mathcal{I}_X$, then there hold the following equivalent conditions:

a) $\forall A \subseteq X~f[A^*]\subseteq (f[A])^*;$

b) $\forall B \subseteq Y~(f^{-1}[B])^*\subseteq f^{-1}[B^*].$

\end{thm}

\begin{proof}
Let us prove that a) holds. Suppose that there exists $A \subseteq X$  such that there exists $y \in f[A^*]\setminus (f[A])^*$. So, there exists $x\in A^*$ such that $f(x)=y$ and
\begin{equation}\label{eqc1}
\forall U \in \tau_X(x)~\quad  U \cap A \not \in \mathcal{I}_X.
 \end{equation}
 Since $y \not \in (f[A])^*$, there exists $V \in \tau_Y(y)$ such that $V \cap f[A] \in \mathcal{I}_Y$. Then, $f^{-1}[V \cap f[A]] \in \mathcal{I}_X$. Since, for each $C,D \in Y$, holds $f^{-1}[C ]\cap f^{-1}[D]  \subseteq f^{-1}[C \cap D]$, by the hereditary property of ideals, we have $f^{-1}[V] \cap f^{-1}[f[A]] \in \mathcal{I}_X$, and since $A \subseteq f^{-1}[f[A]]$, there holds \begin{equation} \label{eqc2}f^{-1}[V] \cap A \in \mathcal{I}_X.
 \end{equation}
 Finally, due to the continuity of $f$, $f^{-1}[V]\in \tau_X(x)$, so \eqref{eqc2} contradicts \eqref{eqc1}, proving a).

 Let us show that b) is equivalent to a). Suppose a) holds and let $B \subseteq Y$.  Then $f[(f^{-1}[B])^*]\subseteq (f[f^{-1}[B]])^*$, and, since $f[f^{-1}[B]]\subseteq B$, then $f[(f^{-1}[B])^*] \subseteq B^*$. Now, by taking the preimage of both sets, we have $f^{-1}[f[(f^{-1}[B])^*]] \subseteq f^{-1}[B^*]$, and, since $f^{-1}[f[C]] \supseteq C$, for each $C\subseteq X$, we obtain $(f^{-1}[B])^* \subseteq f^{-1}[B^*]$.

 Now suppose b) holds and let $A\subseteq X$. Then $f^{-1}[(f[A])^*]\supseteq (f^{-1}[f[A]])^* \supseteq A^*$. By taking the direct image by $f$ of both sets we obtain $f[f^{-1}[(f[A])^*]]\supseteq f[A^*]$. Finally, since $f[f^{-1}[D]] \subseteq D$, for each $D \subseteq Y$, we obtain $(f[A])^*\supseteq f[A^*]$.
\end{proof}

\begin{rem} Condition a) in case of  bijective mappings is, by the result of  Sivaraj and Renuka Devi \cite{MR2247422},   equivalent with pointwise-$\mathcal{I}$-continuity, introduced  by  Kaniewski and Piotrowski \cite{KanPiot}. Also, some interesting sufficient conditions for  a) can be found in the work of Goyal and Noorie \cite{GoyalNoorie}.
\end{rem}

\begin{ex} The opposite does not hold even if for each $I \in \mathcal{I}_Y$ holds $f^{-1}[I] \in \mathcal{I}_X$. If $ \mathcal{I}_X=P(X)$, then for each $A \subseteq X$ we have $A^*=\emptyset$, so a) trivially holds, and also, for each $I \in \mathcal{I}_Y$ we have $f^{-1}[I] \in \mathcal{I}_X=P(X)$. But $f$, in general, does not have to be continuous.
\end{ex}

\begin{thm} \label{Tc2} Let $\langle X, \tau_X, \mathcal{I}_X\rangle$ and $\langle Y, \tau_Y, \mathcal{I}_Y\rangle$ be ideal topological spaces. If $f:\langle X, \tau_X\rangle \to \langle Y, \tau_Y\rangle$ is a continuous function and for all $I\in \mathcal{I}_Y$ we have $f^{-1}[I]\in \mathcal{I}_X$, then there hold the following three equivalent conditions:

a) $\forall A \subseteq X~f[\overline{A}^{\tau_X^*}] \subseteq \overline{f[A]}^{\tau_Y^*}$;

b) $\forall B \subseteq Y~\overline{(f^{-1}[B])}^{\tau_X^*}\subseteq f^{-1}[\overline{B}^{\tau_Y^*}];$

c) $f:\langle X, \tau_X^*\rangle \to \langle Y, \tau_Y^*\rangle$ is a continuous function.

\end{thm}

\begin{proof}
Obviously, all three conditions are equivalent (see Proposition \ref{continuityeq}), so it will be sufficient to prove a), which easily  follows    from a) in the previous theorem. Therefore, we have
 \begin{eqnarray*}
 f[\overline{A}^{\tau_X^*}]&=&f[A \cup A^*]=f[A] \cup f[A^*]\\
 &\subseteq&f[A] \cup (f[A])^*
 =\overline{f[A]}^{\tau_Y^*}.
 \end{eqnarray*}
Just to remark that condition b) can also be easily obtained from condition b) in the previous theorem.
\end{proof}

\begin{ex}  Condition a) (and therefore b)) in Theorem \ref{Tc1} is not equivalent to the continuity of $f:\langle X, \tau_X^*\rangle \to \langle Y, \tau_Y^*\rangle$. Let $f$ be an arbitrary continuous  mapping from
$\langle X, \tau_X\rangle $ to $\langle Y, \tau_Y\rangle$, where both spaces are $T_1$. Let $x_0 \in X$, $f(x_0)=y_0$, $\mathcal{I}_X=\{A \subset X: |A|<\aleph_0, x_0 \not \in A\}$ and $\mathcal{I}_Y=Fin(Y)$. Then $\tau_X^*=\tau_X$ and $\tau_Y^*=\tau_Y$. So, $f:\langle X, \tau_X^*\rangle \to \langle Y, \tau_Y^*\rangle$ is continuous.

We see that $x_0 \in \{x_0\}^*=\{x_0\}$, so $y_0\in f[\{x_0\}^*]=\{y_0\}$. But $U \cap f[\{x_0\}]=\{y_0\}$, for each neighbourhood $U$ of $y_0$, which is in the ideal  $\mathcal{I}_Y$, so $y_0 \not\in (f[\{x_0\}])^*$, implying that condition a) does not hold.

Another example is to consider on $X$ the topology $\tau_X=P(X)$. Then each mapping, no matter what ideal is chosen, is continuous, so choosing the same ideals as before , we also get a counterexample.
\end{ex}

If we assume that $f$ is a bijection, we obtain the following result.

\begin{thm}\label{contpsi} Let $\langle X, \tau_X, \mathcal{I}_X\rangle$ and $\langle Y, \tau_Y, \mathcal{I}_Y\rangle$ be ideal topological spaces. If $f:\langle X, \tau_X\rangle \to \langle Y, \tau_Y\rangle$ is a continuous bijection and for all $I\in \mathcal{I}_Y$ we have $f^{-1}[I]\in \mathcal{I}_X$, then there hold the following equivalent conditions:

a) $\forall A \subseteq X~\Psi(f[A])\subseteq f[\Psi(A)];$

b) $\forall B \subseteq Y~f^{-1}[\Psi(B)]\subseteq \Psi(f^{-1}[B]).$

\end{thm}
\begin{proof}
Let us firstly prove that conditions a) and b) are equivalent.

a) $\Rightarrow$ b) Suppose $B \subseteq Y$. Then, using the fact that $f\circ f^{-1}=\id$, we have $f^{-1}[\Psi(B)]=f^{-1}[\Psi(f[f^{-1}[B]])]\subseteq f^{-1}[f[\Psi(f^{-1}[B])]]=\Psi(f^{-1}[B])$.

b) $\Rightarrow$ a) Suppose $A \subseteq X$. Then, again using  $f\circ f^{-1}=\id$, we have
$\Psi(f[A])=f[f^{-1}[\Psi(f[A])]]\subseteq f[\Psi(f^{-1}[f[A]])]= f[\Psi(A)]$.

Now, let us prove that a) holds. Suppose that there exists $A \subseteq X$ such that $\Psi(f[A])\setminus f[\Psi(A)] \neq \emptyset$. So,
\begin{eqnarray} \label{eq1} \lefteqn{\Psi(f[A])\setminus f[\Psi(A)]}\\
\nonumber&=&(Y \setminus (Y\setminus f[A])^*)\setminus f[X \setminus (X\setminus A)^*] \\
\nonumber&=&(Y \setminus (Y\setminus f[A])^*)\setminus (f[X] \setminus f[(X\setminus A)^*]) \\
\nonumber&=&(Y \setminus (Y\setminus f[A])^*)\setminus (Y \setminus f[(X\setminus A)^*]) \\
\nonumber&=&f[(X\setminus A)^*]\setminus(Y\setminus f[A])^* \\
\nonumber&=&f[(X\setminus A)^*]\setminus(f[X]\setminus f[A])^* \\
&=&f[(X\setminus A)^*]\setminus(f[X\setminus A])^* \neq \emptyset, \nonumber
\end{eqnarray}
but this contradicts condition a) from Theorem \ref{Tc1}. Note that surjection is used in the fact that $f[X]=Y$, and injection for $f[X] \setminus f[A]=f[X \setminus A]$.
\end{proof}

\begin{rem}
This is also a proof that, if $f$ is a bijection, conditions a) and b) from the previous theorem and from Theorem \ref{Tc1} are equivalent.
If the set $B\subset X$ possibly violates condition a) from Theorem \ref{Tc1}, then just write it in the form $B=X\setminus A$ and apply \eqref{eq1}.
\end{rem}

\begin{ex}
Let us prove that if $f$ is not a surjection in Theorem \ref{contpsi}, but other conditions hold, then condition a) (and at the same time b)) does not hold.

Let $f:\langle X, \tau_X\rangle \to \langle Y, \tau_Y \rangle$ be a homeomorphism. Let $\mathcal{I}_X$ and $\mathcal{I}_Y$ be ideals preserved by homeomorphism, i.e.\ $I \in \mathcal{I}_X$ iff $f[I] \in \mathcal{I}_Y$.  Let $Z=Y \cup \{z\}$, where $z$ is a point outside $Y$, $\tau_Z=\{O \cup \{z\} : O \in \tau_Y\}\cup \{\emptyset\}$, and $\mathcal{I}_Z$ the smallest ideal containing $\mathcal{I}_Y$ and $\{z\}$. Finally, let $\tilde f:\langle X, \tau_X\rangle \to \langle Z, \tau_Z \rangle$ be defined as $\tilde{f}(x)=f(x)$. It is easy to see,  since $\tilde f^{-1}[A]=f^{-1}[A]=f^{-1}[A\setminus\{z\}]$ for each $A\subset Z$, we have  that $\tilde f$ is continuous and   for all $I\in \mathcal{I}_Z$ we have $f^{-1}[I]\in \mathcal{I}_X$.

Then, for each set $A \subset Z$ we have $z \not \in A^*$, i.e., for each $A$ we have $z \in \Psi(A)$, so $z \in \Psi(f[A]))$, but since  $z \not \in f[X]$ we have $ z \not \in  f[\Psi(A)]$.
\end{ex}

\begin{ex}
Let us prove that if $f$ is not an injection in Theorem \ref{contpsi}, and other conditions hold, then condition a) (and at the same time b)) does not hold.

Let $f:\langle X, \tau_X\rangle \to \langle Y \tau_Y \rangle$ be a homeomorphism. Let $\mathcal{I}_X$ and $\mathcal{I}_Y$ be ideals preserved by homeomorphism, i.e.\ $I \in \mathcal{I}_X$ iff $f[I] \in \mathcal{I}_Y$.
Let $y_0 \in Y$ be such that $\{y_0\}\in \mathcal{I}_Y$ and let $x_0$ be such that $f(x_0)=y_0$. Such $y_0$ always exists in case $\mathcal{I}_Y \neq \{\emptyset\}$.  Let $Z=X \cup \{z\}$, where $z$ is a point outside of $X$ and $\tau_Z=\{O : O \not \in \tau_X(x_0)\} \cup \{O \cup \{z\} : O \in \tau_X(x_0)\}$.
Then let $\mathcal{I}_Z=\{I\setminus \{x_0\}: I \in \mathcal{I}_X\}$   and
$\widetilde{\mathcal{I}_Y}=\{I\setminus \{y_0\}: I \in \mathcal{I}_Y\}$.

Finally, we define $\tilde f:\langle Z, \tau_Z,  \mathcal{I}_Z\rangle \to \langle Y, \tau_Y,  \widetilde{\mathcal{I}_Y} \rangle$ by $\tilde{f}(x)=f(x)$, for $x \in X$, and $f(z)=y_0$.

Let us prove that $\tilde f$ fulfils all conditions from Theorem \ref{contpsi} except injectivity.

It is easy to see that $\tilde f^{-1}[A]=f^{-1}[A]$ when $y_0 \not  \in A$, and $\tilde f^{-1}[A]=f^{-1}[A]\cup\{z\}$, if $y_0 \in A$. Then, for $O \in \tau_Y$, $\tilde f^{-1}[O]\in \tau_Z$, and, for each $I \in \tilde{\mathcal{I}_Y}$ we have $\tilde f^{-1}[I]=f^{-1}[I] \in \mathcal{I}_X \subset \mathcal{I}_Z$.

For $A=X$ we have $\tilde f[A]=Y$.   Also, $\Psi(Y)=Y  \setminus (Y \setminus Y)^*=Y$, so $y_0 \in  \Psi(\tilde  f[X])$. On the other hand, $\Psi(X)=Z \setminus \{z\}^*$. Since $\{z\} \not \in \mathcal{I}_Z$, and since each neighbourhood of $z$ and $x_0$ contains $z$, we have $x_0, z \in \{z\}^*$, implying $\Psi(X) \subseteq X \setminus \{x_0\}$, so $\tilde  f[\Psi(X)] \subseteq \tilde f[X \setminus \{x_0\}]=Y \setminus \{y_0\}$, therefore $y_0 \not \in \tilde f [\Psi(X)])$.
\end{ex}

%\begin{ex} wrong Let $X=Y=\omega$, and let, on $\tau_Y=P(\omega)$ and $\tau_X=P(\omega\setminus\{0\}) \cup \{\omega \setminus K: K \subset \omega\setminus \{0\}, |K|< \aleph_0\}$. Let $f(n)=n$. Obviously, $f:\langle \omega, \tau_X\rangle \to \langle \omega, \tau_Y\rangle$ is not continuous, but, it is bijection. Let $\mathcal{I}_X=\mathcal{I}_Y=Fin$ \end{ex}

\section{Local function and open and closed mappings}

In this section we will deal  with  open and closed mappings and their preservation in spaces obtained by the  local function.

\begin{thm} \label{To1} Let $\langle X, \tau_X, \mathcal{I}_X\rangle$ and $\langle Y, \tau_Y, \mathcal{I}_Y\rangle$ be ideal topological spaces. If $f:\langle X, \tau_X\rangle \to \langle Y, \tau_Y\rangle$ is an open function and for all $I\in \mathcal{I}_X$ we have $f[I]\in \mathcal{I}_Y$, then there hold the \textit{following} equivalent conditions:

a) $\forall A \subseteq X~f[\Psi(A)]\subseteq \Psi(f[A]);$

b) $\forall B \subseteq Y~\Psi(f^{-1}[B])\subseteq f^{-1}[\Psi(B)]$.

\end{thm}

\begin{proof}Firstly, let us prove a). For $A \subseteq X$ suppose that there exists $y \in f[\Psi(A)]\setminus \Psi(f[A])$. Then there exists $x\in \Psi(A)$ such that $f(x)=y$. So, $x \in X \setminus (X\setminus A)^*$, implying $x \not \in  (X\setminus A)^*$. Therefore, there exists $U\in \tau_X(x)$ such that $U \cap (X\setminus A)\in \mathcal{I}_X$, i.e.\ $U \setminus A\in \mathcal{I}_X$. This, by the property of the  ideals, gives $f[U\setminus A] \in \mathcal{I}_Y$, and since $f[U]\setminus f[A]\subseteq f[U\setminus A]$, we have
\begin{equation}\label{eqo1}
f[U]\setminus f[A] \in \mathcal{I}_Y.
\end{equation}

Since $f(x)  \not \in \Psi(f[A])=Y \setminus (Y\setminus f[A])^*$, we have $f(x) \in (Y\setminus f[A])^*$, and for each $V \in \tau_Y(f(x))$ we have $V \setminus f[A] \not \in \mathcal{I}_Y$. But, since $f$ is an open mapping, $f[U]$ is open, which gives $f[U] \setminus f[A] \not \in \mathcal{I}_Y$, contradicting \eqref{eqo1}.

Now, let us show that a) implies b). Suppose that, for all $A \subseteq X$ we have  $ f[\Psi(A)]\subseteq \Psi(f[A])$. Therefore $ f[\Psi(f^{-1}[B])]\subseteq \Psi(f[f^{-1}[B]])\subseteq \Psi(B)$.  Applying the preimage on both sets, and from the fact that $f^{-1}[f[C]]\supseteq C$, we get
$\Psi(f^{-1}[B])\subseteq  f^{-1}[f[\Psi(f^{-1}[B])]]\subseteq f^{-1}[\Psi(B)]$.

If b) holds, then $f^{-1}[\Psi(f[A])]\supseteq \Psi(f^{-1}[f[A]]) \supseteq \Psi(A)$. Applying $f$, and due the fact that $f[f^{-1}[D]]\subset D$ we get
$ f[\Psi(A)]\subseteq f[f^{-1}[\Psi(f[A])]]\subseteq \Psi(f[A])$, which completes the proof.
\end{proof}

\begin{thm} Let $\langle X, \tau_X, \mathcal{I}_X\rangle$ and $\langle Y, \tau_Y, \mathcal{I}_Y\rangle$ be ideal topological spaces. If $f:\langle X, \tau_X\rangle \to \langle Y, \tau_Y\rangle$ is an open function and for all $I\in \mathcal{I}_X$ we have $f[I]\in \mathcal{I}_Y$, then  $f:\langle X, \tau_X^*\rangle \to \langle Y, \tau_Y^*\rangle$ is an open function.
\end{thm}
\begin{proof} Set $A$ is open in $\tau^*_X$ iff $A \subseteq \Psi(A)$. So, for such $A$ we have
$$f[A] \subseteq f[\Psi(A)]\subseteq \Psi(f[A]),$$
implying $f[A]$ is  in $\tau^*_Y$.
\end{proof}

\begin{ex} Property that $f:\langle X, \tau_X^*\rangle \to \langle Y, \tau_Y^*\rangle$ is an open function is not equivalent with condition a) (and b)) from Theorem \ref{To1}.
Let $f$ be an arbitrary open surjection from
$\langle X, \tau_X\rangle $ to $\langle Y, \tau_Y\rangle$, where both spaces are $T_1$. Let $x_0 \in X$ and let   $f^{-1}[\{f(x_0)\}]=\{x_0\}$. We take $\mathcal{I}_X=Fin(X)$ and $\mathcal{I}_Y= \{B \subset X: |B|<\aleph_0, f(x_0) \not \in B\}$. Then, due to $T_1$, ideals do not change topologies on $X$ and on $Y$, so the mapping remains open.

Let $A=X \setminus \{x_0\}$. Then $\Psi(A)=X \setminus (X\setminus (X\setminus \{x_0\}))^*=X \setminus (\{x_0\})^*=X$. So $f[\Psi(A)]=f[X]=Y$. On the other hand $\Psi(f[A])=\Psi(f[X]\setminus \{f(x_0)\})=Y \setminus (Y\setminus (Y\setminus \{f(x_0)\}))^*
=Y \setminus \{f(x_0)\}^*\subset  Y \setminus \{f(x_0)\}$.
\end{ex}

\begin{thm} \label{openbij} Let $\langle X, \tau_X, \mathcal{I}_X\rangle$ and $\langle Y, \tau_Y, \mathcal{I}_Y\rangle$ be ideal topological spaces. If $f:\langle X, \tau_X\rangle \to \langle Y, \tau_Y\rangle$ is an open bijection and for all $I\in \mathcal{I}_X$ we have $f[I]\in \mathcal{I}_Y$, then there hold the \textit{following} equivalent conditions:

a) $\forall A \subseteq X~ (f[A])^* \subseteq  f[A^*];$

b) $\forall B \subseteq Y~ f^{-1}[B^*]\subseteq (f^{-1}[B])^*.$

\end{thm}

\begin{proof} Firstly we will prove the equivalence of a) and b).

 a) $\Rightarrow$ b) Suppose $B \subseteq Y$. Then, using that $f\circ f^{-1}=\id$, we have $f^{-1}[B^*]=f^{-1}[(f[f^{-1}[B]])^*]\subseteq f^{-1}[f[(f^{-1}[B])^*]]=(f^{-1}[B])^*$.

b) $\Rightarrow$ a) Suppose $A \subseteq X$. Then, again using that $f\circ f^{-1}=\id$, we have
$(f[A])^*=f[f^{-1}[(f[A])^*]]\subseteq f[(f^{-1}[f[A]])^*]= f[A^*]$.

Now, let us prove that a) holds. Suppose that there exists $A \subseteq X$ such that $(f[A])^*\setminus f[A^*] \neq \emptyset$. So,
\begin{eqnarray} \label{eq2} \lefteqn{(f[A])^*\setminus f[A^*]}\\
\nonumber&=& {(f[X \setminus (X \setminus A)])^*\setminus f[(X \setminus (X \setminus A))^*]} \\
\nonumber&=& {(f[X] \setminus f[(X \setminus A)])^*\setminus f[(X \setminus (X \setminus A))^*]} \\
\nonumber&=& {(Y \setminus f[(X \setminus A)])^*\setminus f[(X \setminus (X \setminus A))^*]} \\
\nonumber&=& {(Y\setminus f[(X \setminus (X \setminus A))^*]) \setminus (Y \setminus(Y \setminus f[(X \setminus A)])^*) } \\
\nonumber&=& { f[\Psi(X \setminus A)] \setminus \Psi(f[X \setminus A]) }  \neq \emptyset,
\end{eqnarray}
since $f[\Psi(X \setminus A)]=(Y\setminus f[(X \setminus (X \setminus A))^*])$, which contradicts  a) from Theorem \ref{To1}.

\end{proof}

\begin{ex}
If $f$ is not a surjection in Theorem \ref{openbij}, and other conditions hold, then condition a) (and at the same time b)) does not have to hold.

Let $f:\langle X, \tau_X\rangle \to \langle Y \tau_Y \rangle$ be a homeomorphism. Let $\mathcal{I}_X$ and $\mathcal{I}_Y$ be ideals preserved by homeomorphism, i.e.\ $I \in \mathcal{I}_X$ iff $f[I] \in \mathcal{I}_Y$.  Let $Z=Y \cup \{z\}$, where $z$ is a point outside $Y$, $\tau_Z=\langle \tau_Y\cup\{Z\}\rangle$,%
%\{O \cup \{z\} : O \in \tau_Y\}\cup \{\emptyset\}$,
and $\mathcal{I}_Z=\mathcal{I}_Y$. So, $\{z \} \not \in \mathcal{I}_Z$ and $\tau_Z(z)=\{Z\}$.
 Finally, let $\tilde f:\langle X, \tau_X\rangle \to \langle Z ,\tau_Z \rangle$ be defined by $\tilde{f}(x)=f(x)$. It is easy to see, since $\tilde f[A]=f[A]$ for each $A\subset X$, we have  that $\tilde f$ is open and   for all $I\in \mathcal{I}_X$ we have $f[I]\in \mathcal{I}_Z$.

Then, for each set $A \in P(Z)\setminus \mathcal{I}_Z$ we have $z \in A^*$. So, if we take any ideal on $X$ different from $P(X)$, we have $z\in (f[X])^*$, but $z\not \in f[X^*]$.
\end{ex}

\begin{ex}\label{opennotinj}

If $f$ is not an injection in Theorem \ref{openbij}, and other condition hold, then condition a) (and at the same time b)) does not have to  hold.

Let $f:\langle X, \tau_X\rangle \to \langle Y ,\tau_Y \rangle$ be a homeomorphism. Let $\mathcal{I}_X$ and $\mathcal{I}_Y$ be ideals preserved by homeomorphism, i.e.\ $I \in \mathcal{I}_X$ iff $f[I] \in \mathcal{I}_Y$.
Let $y_0 \in Y$ be such that $\{y_0\}\in \mathcal{I}_Y$ and let $x_0$ be such that $f(x_0)=y_0$. Such $y_0$ always exists in case $\mathcal{I}_Y \neq \{\emptyset\}$.  Let $Z=X \cup \{z\}$, where $z$ is a point outside of $X$ and $\tau_Z=\{O : O \not \in \tau_X(x_0)\} \cup \{Z\}$.
Then let $\mathcal{I}_Z=\{I\setminus \{x_0\}: I \in \mathcal{I}_X\}$   and
$\tilde{\mathcal{I}_Y}=\{I \setminus  \{y_0\}: I \in \mathcal{I}_Y\}$.

Finally, let $\tilde f:\langle Z, \tau_Z,  \mathcal{I}_Z\rangle \to \langle Y, \tau_Y,  \tilde{\mathcal{I}_Y} \rangle$ is defined by $\tilde{f}(x)=f(x)$, for $x \in X$, and $f(z)=y_0$.

Let us prove that $\tilde f$ fulfils all conditions from Theorem \ref{openbij} except injectivity.
It is easy to see that $\tilde f[A]=f[A]$ when $z \not  \in A$, and $\tilde f[Z]=f[X]=Y$. So, $\tilde f$ remains open. For each $I \in \tilde{\mathcal{I}_X}$ we have $\tilde f[I]=f[I_1\setminus\{x_0\}]=f^{-1}[I_1]\setminus \{y_0\} \in  \mathcal{I}_Y$, where $I_1$ is some set from ideal $\mathcal{I}_X$.

For $A=X$ we have $\tilde f[A]=Y$, and, since $\{y_0\} \not \in \tilde{ \mathcal{I}_Y}$, $y_0 \in (\tilde f[A])^*$. On the other hand, since the only neighbourhood of $z$ an $x_0$ is whole $Z$, if we take $\mathcal{I}_X$ such that $X \setminus \{x_0\}$ is in the starting ideal, then $A^*=X\setminus \{x_0\}$, so $y_0 \not \in f[A^*]$.
\end{ex}

If we consider closed mappings, since each open bijection is, at the same time a closed bijection,  we can freely change the word "open" with "closed" in Theorem \ref{openbij}. However, in such case, we can improve our result.

\begin{thm} \label{closedsur} Let $\langle X, \tau_X, \mathcal{I}_X\rangle$ and $\langle Y, \tau_Y, \mathcal{I}_Y\rangle$ be ideal topological spaces. Let $f:\langle X, \tau_X\rangle \to \langle Y, \tau_Y\rangle$ be a closed injection and for all $I\in \mathcal{I}_X$ let $f[I]\in \mathcal{I}_Y$. Then there hold the \textit{following} equivalent conditions:

a) $\forall A \subseteq X~ (f[A])^* \subseteq  f[A^*];$

b) $\forall B \subseteq Y~ f^{-1}[B^*]\subseteq (f^{-1}[B])^*.$

\end{thm}
\begin{proof}
Since $X$ is a closed set and $f$ a closed mapping,  $f[X]$  is also closed,  which implies that $Y \setminus  f[X]$ is open in $Y$. So, if we take the surjective restriction $g:X \to f[X]$, such that $g(x)=f(x)$, we obtain a closed bijection. By Theorem \ref{openbij} and the remark before this theorem, we have  $\forall A \subseteq X~ (g[A])^* \subseteq  g[A^*]$. Obviously, $g[A^*]=f[A^*]$. But also, $(g[A])^*$ in space $f[X]$ is equal to $(f[A])^*$ in $Y$, since, for each $y \in Y \setminus f[X]$, the set $Y \setminus f[X]$ is an open neighbourhood of $y$ which does not intersect  $f[A]$. Therefore $(f[A])^*=(g[A])^* \subseteq  g[A^*]=f[A^*]$.
\end{proof}

\begin{ex}
In Example \ref{opennotinj} the given mapping is, in the same time, a closed surjection, so it witnesses that Theorem \ref{closedsur} can not be further weakened.
\end{ex}

It remains an open question if we can state some nice property for closed noninjective mappings. Example \ref{opennotinj} shows that even  closed "finite to one" mappings does not imply  results from Theorem \ref{closedsur}.

\vskip 1cm

Finally, gathering all previous, we extended the result obtained by Janković and Hamlett \cite{HRStar}, which was already mentioned in Theorem \ref{homeo}.

\begin{cor}
Let $\langle X, \tau_X, \mathcal{I}_X\rangle$ and $\langle Y, \tau_Y, \mathcal{I}_Y\rangle$ be ideal topological spaces. If $f:\langle X, \tau_X\rangle \to \langle Y, \tau_Y\rangle$ is homeomorphism and for each $I \subset X$ there holds  $I\in \mathcal{I}_X$ iff $f[I]\in \mathcal{I}_Y$. Then the following equivalent conditions hold:

a) $f:\langle X, \tau_X^*\rangle \to \langle Y, \tau_Y^*\rangle$ is a homeomorphism;

b) $\forall A \subseteq X~ (f[A])^* = f[A^*];$

c) $\forall B \subseteq Y~ f^{-1}[B^*]= (f^{-1}[B])^*.$

d) $\forall A \subseteq X~\Psi(f[A])= f[\Psi(A)];$

e) $\forall B \subseteq Y~f^{-1}[\Psi(B)]= \Psi(f^{-1}[B]).$
\end{cor}

\end{document}